\documentclass[11pt]{article}      
\usepackage{amsmath,amsthm,amsfonts,graphicx}
\usepackage{multirow}
\usepackage{multirow}
\def\smallddots{\mathinner{\raise7pt\hbox{.}\raise4pt\hbox{.}\raise1pt\hbox{.}}} 
\def\smallsdots{\mathinner{\raise1pt\hbox{.}\raise4pt\hbox{.}\raise7pt\hbox{.}}}

\DeclareMathOperator{\diag}{diag}

\DeclareMathOperator*{\argmin}{arg\,min}

\newtheorem{theorem}{Theorem}[section]
 
\numberwithin{equation}{section}
\numberwithin{table}{section}

\newtheorem{algorithm}{Algorithm}[section]

\newtheorem{problem}{Problem}[section]
\begin{document} 
 
\title{Superfast Approximate \\
 Least Squares Solution of a Highly \\ Overdetermined 
 Linear System of Equations} 
\author{Victor Y. Pan} 
\author{Qi Luan$^{[1],[a]}$  and 
Victor Y. Pan$^{[1,2],[b]}$\\
$^{[1]}$ Ph.D. Programs in  Computer Science and Mathematics \\
The Graduate Center of the City University of New York \\
New York, NY 10036 USA \\
$^{[2]}$ Department of Computer Science \\
Lehman College of the City University of New York \\
Bronx, NY 10468 USA \\
$^{[a]}$ qi\_luan@yahoo.com \\ 
$^{[b]}$ victor.pan@lehman.cuny.edu \\ 
http://comet.lehman.cuny.edu/vpan/  \\
}

\date{} 

\maketitle


\begin{abstract}
With a high probability ({\em whp})
the Sarl\`os randomized algorithm of 2006 outputs a nearly optimal least squares solution of a highly overdetermined  
 linear system of equations.
 We propose its simple deterministic variation that computes such a solution   
 for a random input whp
and therefore computes it deterministically for a large input 
class. Unlike the Sarl\`os
 original algorithm our variation performs computations at {\em sublinear cost} or, as we say, {\em superfast},
that is, by using much fewer memory cells and arithmetic operations than an input matrix has entries.
 Our extensive tests are in good accordance with this result.
 \end{abstract}
 


 

 
\section{Introduction}\label{sintro}

{\bf The Linear Least Squares Problem (LLSP).} The computation of a least squares solution
 to a highly overdetermined 
 linear system of equations is a  hot research subject, fundamental for Matrix Computations and Big Data Mining and Analysis (see the books \cite{LH95} and \cite{B96},
  more recent works \cite{DMM06},  \cite{RT08}, \cite{BD09}, \cite{AMT10},
 \cite{DMMS11},  \cite{M11}, \cite{LWM18},   and the references therein). By  following {Bj{\"o}rk's book \cite{B15} we call this task  Linear Least Squares Problem and hereafter refer to it as
 {\em LLSP}. Some authors call it differently, e.g., 
 {\em Least Squares Approximation  Problem} \cite{M11}.
  
   The matrices that define Big Data 
are frequently so immense that realistically 
one can  access only a tiny fraction of their entries and 
thus must  
perform computations at sublinear cost or, as we say,
{\em superfast} -- by using much fewer    
  memory cells and arithmetic operations 
than an input matrix has 
entries.

 {\bf Our progress.} Unfortunately the output of any   superfast algorithm  is
 extremely far from the solution to LLSP  problem on the worst case inputs (see the beginning of Section \ref{sdual}), but we 
  circumvent this problem 
 based on a simple new application of the Sarl\`os randomized algorithm of \cite{S06}.
 This algorithm  
  multiplies an input matrix  of LLSP
 by a rectangular multiplier, so that the resulting LLSP has a much smaller size and 
 can be readily solved superfast.
 The original input matrix 
  can be assumed to be orthogonal without loss of generality, and for the multiplier 
  one can choose a rectangular Gaussian random matrix, that is, a matrix  filled with  independent identically distributed Gaussian    (aka normal) random variables. Hereafter we  call such a matrix just {\em Gaussian}. 
  
  Its product with an orthogonal
 matrix is Gaussian, by virtue of orthogonal invariance, and
 Sarl\`os proves that the solution of  
  the resulting LLSP closely approximates
  the solution of the original LLSP whp.
 Now we observe that his result still holds 
 for the solution of the {\em dual LLSP}, with a Gaussian input matrix and any orthogonal multiplier (see our Theorem \ref{thLLSPdd}). 
 
 Is this result of any interest?  Yes, because by fixing a sparse
 multiplier we  devise  a
{\em superfast deterministic algorithm} that 
computes a nearly optimal solution to LLSP
for a large class of highly overdetermined inputs. 

We say ``a large class", but could have argued that in a sense this holds for most inputs, because the nearly optimal solution 
is obtained for Gaussian input whp and because such a claim is in a good accordance 
with the results of our tests for both synthtic and real world inputs.
 
 Of course one would prefer to solve LLSP  
  superfast for all inputs, but
 
 (a) this is impossible,
 
 (b) the user should be satisfied even if
 the problem is solved just for the input   
 class of her/his interest,
 
 (c) Theorem \ref{thLLSPdd} implies that this outcome is quite likely for many applications,  and
    
    (d) the results of our extensive tests with both synthetic and real world inputs are in good accordance with that theorem.

 For a simple summary: {\em  it is impossible to ensure superfast solution for ANY INPUT,
 but we describe a simple superfast algorithm that works for MANY INPUTS, and we provide evidence to this by proving that the algorithm works whp for a random input and by supporting it with extensive numerical tests.}
 
{\bf Related works.}
Our present result is just a new example of dual matrix algorithms, which we earlier 
devised, analyzed, and tested for other fundamental matrix computations such as Gaussian elimination with no pivoting
 in \cite{PQY15} and \cite{PZ17} and   Low Rank Approximation (LRA) of a 
 matrix in \cite{PZ16}, \cite{PLSZ16}, \cite{PLSZ17},
  \cite{PLSZ20}, \cite{PLSZa}, \cite{PLa},  
 \cite{LPSa}, and \cite{Pa}.\footnote{The papers \cite{PLSZ16} and
 \cite{PLSZ17}  provide first formal support for  
superfast (dual) LRA.}
We hope that these studies together with our present work will motivate research efforts towards devising superfast deterministic algorithms for matrix computations that would work for  large classes of inputs, thus providing an alternative to the known randomized algorithms that   
 whp solve
these probems for all inputs
but perform much slower.

 {\bf Organization of the paper.}
 In the next section we recall the LLSP and  its  randomized
  approximate solution by Sarl\`os of \cite{S06}.  We present our superfast variation of his algorithm
   in  Section \ref{sdual}. In Section \ref{ststs}, the contribution  of the first author,
  we cover our numerical tests.

 
\section{Randomized Approximate Solution of an Overdetermined Linear System of Equations}\label{sext} 
     

\begin{problem}\label{pr1} {\em [Least Squares Solution of an Overdetermined Linear System of Equations (LLSP).]}
Given two integers $m$ and $d$ such that $1\le d< m$,
a matrix $A\in \mathbb R^{m\times d}$, and a vector ${\bf b}\in \mathbb C^{m}$,
compute a vector  
${\bf x}\in \mathbb R^{d}$ that minimizes the spectral norm $||A{\bf x}-{\bf b}||$
or equivalently  computes the subvector
${\bf x}=(y_j)_{j=0}^{d-1}$ of the vector 
\begin{equation}\label{eqLLSPmy}
{\bf y}=(y_j)_{j=0}^d={\rm argmin}_{\bf v}~
||M{\bf v}||~{\rm such~that}~M=(A~|~{\bf b})~
{\rm and}~
{\bf v}=\begin{pmatrix}  {\bf x} \\
 -1
\end{pmatrix}. 
\end{equation} 
\end{problem}

The minimum norm solution to this problem is given by the vector ${\bf x}=A^+{\bf b}$
for $A^+$ denoting  the Moore--Penrose pseudo inverse of $A$; 
 $A^+{\bf b}=(A^*A)^{-1}A^*{\bf b}$ if a matrix $A$ has full rank $d$.



 
\begin{algorithm}\label{algapprls} 
({\rm  Randomized Approximate Solution of  LLSP from \cite{S06}.})

 
\begin{description}

 
\item[{\sc Input:}] 
An $m\times (d+1)$  matrix $M$.

\item[{\sc Output:}]  
A vector ${\bf x}\in \mathbb R^{d}$
 approximating a solution 
of Problem \ref{pr1}.


\item[{\sc Initialization:}] 
 Fix an 
 integer $s$ such that $d\le s\ll m$. 


\item[{\sc Computations:}]

\begin{enumerate}
\item 
Generate a matrix  $F\in \mathbb R^{s\times m}$.
\item 
Compute a solution ${\bf x}$ of Problem \ref{pr1}
for the $s\times (d+1)$ matrix $FM$. 
\end{enumerate}


\end{description}


\end{algorithm}

 
 Clearly the transition to an input matrix 
 $FM$ simplifies Problem \ref{pr1}
 because its size decreases, and the simplification is dramatic where $s\ll m$, while  
 the following theorem shows that the algorithm still outputs nearly optimal approximate solution to Problem \ref{pr1} for $M$ whp
 if ${\sqrt s}~F$ is in the linear space
 $\mathcal G^{s\times m}$ of $s\times m$ Gaussian   
 matrices.\footnote{Such an approximate  solution
serves as a pre-processor for practical implementation of
numerical linear algebra algorithms 
for Problem \ref{pr1} of least squares computation \cite[Section 4.5]{M11}, \cite{RT08}, \cite{AMT10}.}

 
\begin{theorem}\label{thLLSPp}{\rm (Error Bound for Algorithm \ref{algapprls}. See \cite{S06} or \cite[Theorem 2.3]{W14}.)}   
Let us be given 
two integers $s$ and  $d$
 such that $0<d\le s$, two  
 matrices $M\in \mathbb R^{m\times (d+1)}$ and Gaussian $F \in \mathcal G^{s\times m}$,
 and  two tolerance values $\gamma$ and  
$\epsilon$  such that  
\begin{equation}\label{eqxigm}
0<\gamma<1,~ 
0<\epsilon<1,~{\rm and}~s=\Big(d+\log\Big(
\frac{1}{\gamma}\Big)\Big)~\frac{\eta}{\epsilon^{2}}
\end{equation}
 for a constant $\eta$. 
Then  
\begin{equation}\label{eqLLSPp}
{\rm Probability}\Big\{1-\epsilon\le 
\frac{1}{\sqrt s}~\frac{||FM{\bf y}||}{||M{\bf y}||}\le 1+\epsilon~{\rm for~all~vectors}~
  {\bf y}\neq {\bf 0}\Big\}\ge 1-\gamma.
\end{equation}
\end{theorem}

  
 The computation of the matrix $FM$
 involves order of $dsm\ge d^2m$ flops;  for $m\gg s$
this dominates the overall arithmetic computational cost  of the solution
of Problem \ref{pr1}.  

The current record upper estimate for this cost is  $O(d^2m)$
 (see \cite{CW17}, \cite[Section 2.1]{W14}), while 
 the record lower bound of \cite{CW09} 
 has  order $ (m + d)s\epsilon^{-1} \log(md)$ provided that the relative output error norm is within a
factor  of $1+\epsilon$ from its minimal value. 


\section{Superfast Dual LLSP}\label{sdual}  
  

Any   superfast algorithm for LLSP   misses an input entry $m_{i,j}$ for some pair  $i$ and $j$; therefore for its output the  norm $||M{\bf y}||$ exceeds the norm 
$||M{\bf x_{\rm optimal}}||$by arbitrarily
large factor for the worst case input $M$. Indeed 
modification of  $m_{i,j}$ does not change the output of such an algorithm but can dramatically change an
optimal solution to the
 LLSP.\footnote{Here we assume that
 $y_j\neq 0$. Otherwise we could delete the  $j$th column of $M$, thus decreasing the input size.}
 
The argument can be immediately extended to
randomized algorithms, and so no  superfast deterministic or randomized algorithm can solve Problem \ref{pr1} for
all inputs.  E.g., the  randomized Algorithm \ref{algapprls} 
outputs    nearly optimal solution of LLSP whp for any input, but it is not superfast.
Now we can repeat our argument in the Introduction that the user would be happy even if  the problem is solved just for a class of inputs of her/his interest, and this   
 should motivate devising superfast algorithms that would solve LLSP for a large class of inputs. 
Our next theorem implies that a simple variation of  Algorithm \ref{algapprls} is quite likely to do this. 

Namely, by virtue of that theorem 
the algorithm still outputs a nearly optimal solution of Problem \ref{pr1}
whp in the case of a random Gaussian input $M$ and {\em any} properly scaled unitary multiplier $F$.  By choosing a proper sparse multiplier we
arrive at a superfast
algorithm.

 
\begin{theorem}\label{thLLSPdd} {\em [Error Bounds for Dual LLSP.]}
Suppose that we are given three  integers $s$, $m$, and  $d$ such that $0<d\le s<m$, and
two tolerance values $\gamma$ and  
$\epsilon$  satisfying (\ref{eqxigm}).
Define  a unitary matrix 
 $Q_{s,m}\in  \mathbb R^{s\times m}$
 and
a matrix $G_{m,d+1}\in \mathcal G^{m\times (d+1)}$ and write
\begin{equation}\label{eqfm}
 F:=a~Q_{s,m}~{\rm and}~M:=b~G_{m,d+1}
 \end{equation}
 for two scalars $a$ and $b$ such that
 $ab\sqrt s=1$.
Then
$${\rm Probability}\Big\{1-\epsilon~\le \frac{||FM{\bf z}||}{||M{\bf z}||}\le 1+\epsilon~{\rm for~all~vectors}~
  {\bf z}\neq {\bf 0}\Big\}\ge 1-\gamma.$$
\end{theorem}
\begin{proof} 
The claim follows from Theorem \ref{thLLSPp} because  the $s\times (d+1)$  matrix 
$\frac{1}{ab}FM$ is Gaussian by virtue of unitary invariance of Gaussian matrices.
\end{proof}
 
The theorem shows that for a Gaussian matrix $M$ and any properly scaled unitary matrix $F$ of (\ref{eqfm}) 
 Algorithm \ref{algapprls}  outputs a solution of Problem \ref{pr1} that whp  is  optimal up to 
a factor lying in the range  $[1-\epsilon,1+\epsilon]$.
     

\section{Numerical Tests for LLSP}\label{ststs}


In this section we present the results of our tests  of  Algorithm \ref{algapprls}
for both synthetic and real-world data.
We worked with 
 random unitary
multipliers, let 
${\bf x}: =\argmin_{\bf u}||FA{\bf u}- F{\bf b}||$, and computed the relative residual norm 
$$\displaystyle \frac{|| A{\bf x} - {\bf b}||}{\min_{\bf u} ||A{\bf u}-{\bf b}||}.$$ 
In our tests these  ratios quite closely approximated 1 from above.

We used the following random scaled unitary multipliers $F\in \mathbb{R}^{s\times m}$:  

(i) full rank $s\times m$ submatrices of $m\times m$ 
random permutation matrices,

(ii) ASPH matrices from \cite{PLSZ20} and \cite{PLSZa}, which we output after performing just the first three recursive steps out of $\log_2{m}$ steps involved into the generation of the  matrices 
of subsampled randomized Hadamard thansform, and 

(iii) block permutation matrices formed by filling  $s\times m$ matrices with $c=m/s$ identity matrices, each of size $s\times s$, and performing random column permutations; we have chosen $c = 8$ to match the computational cost of the application of ASPH multipliers.
  
 For comparison we also included the
 test results with $s\times m$ Gaussian multipliers.

 We performed our tests on a machine with Intel Core i7 processor running Windows 7 64bit; 
we invoked the \textit{lstsq} function from Numpy 1.14.3 for solving the LLSPs.

\subsection{Synthetic Input Matrices}
 
 For synthetic inputs, we generated input matrices $A\in \mathbb{R}^{m\times d}$   by following
 (with a few modifications) the  recipes of extensive tests in \cite{AMT10},  which compared the running time 
of the regular LLSP problems and the reduced ones with WHT, DCT, and DHT pre-processing. 

We used input
matrices $A$  of size $4096 \times 50$ and $16834 \times 100$ being either
Gaussian matrices or random ill-conditioned matrices.
 We generated the input vectors
 ${\bf b} =\frac{1}{||A{\bf w}||} A{\bf w} +
 \frac{0.001}{||{\bf v}||}{\bf v}$,
 where ${\bf w}$ and  ${\bf v}$ were random
Gaussian vectors of size $d$ and $m$, respectively, 
and so  
 ${\bf b}$ 
was
in the range of $A$
up to a smaller term 
$\frac{0.001}{||{\bf v}||}{\bf v}$.
 
Figure \ref{LLSPRandom} displays the test results for  Gaussian input matrices. 

Figure \ref{LLSPill} displays the test results for  ill-conditioned random inputs 
defined by their SVD $A = U\Sigma V^*$ where 
 the unitary matrices 
  $U$ and $V$  of singular vectors were given by the Q factors in 
 QR-factorization of 
 two independent Gaussian matrices
 and where $\Sigma=\diag(\sigma_j)_j$ with 
  $\sigma_j=10^{5-j}$ for $j = 1, 2\dots , 14$ and $\sigma_j=10^{-10}$ for $j>14$.


Our input matrices $A$ were  highly over-determined, having many more rows than columns.  
 We have chosen $s = dh$, $h= 2, 3,  4, 5, 6$ for 
the multipliers $F$.   By decreasing the ratio $h = s/d$ and therefore the integer $s$ we would
 accelerate the computations of our algorithm,
but we had
to keep it large enough in order to yield
accurate
 solution.

We performed 100 tests with 100 random multipliers for every triple of the input class, multiplier class, and test sizes
(cf. (\ref{eqxigm})) and computed
the  mean of the 100 relative residual norms of the output.
 
We  display the test results
in Figures \ref{LLSPRandom} and \ref{LLSPill} with  ratio $h$ marked on the horizontal axis. 
The tests show  that our multipliers were consistently
effective for random matrices.   The performance was not affected 
by 
the 
conditioning of input matrices.

\begin{figure}
\includegraphics[width = \textwidth]{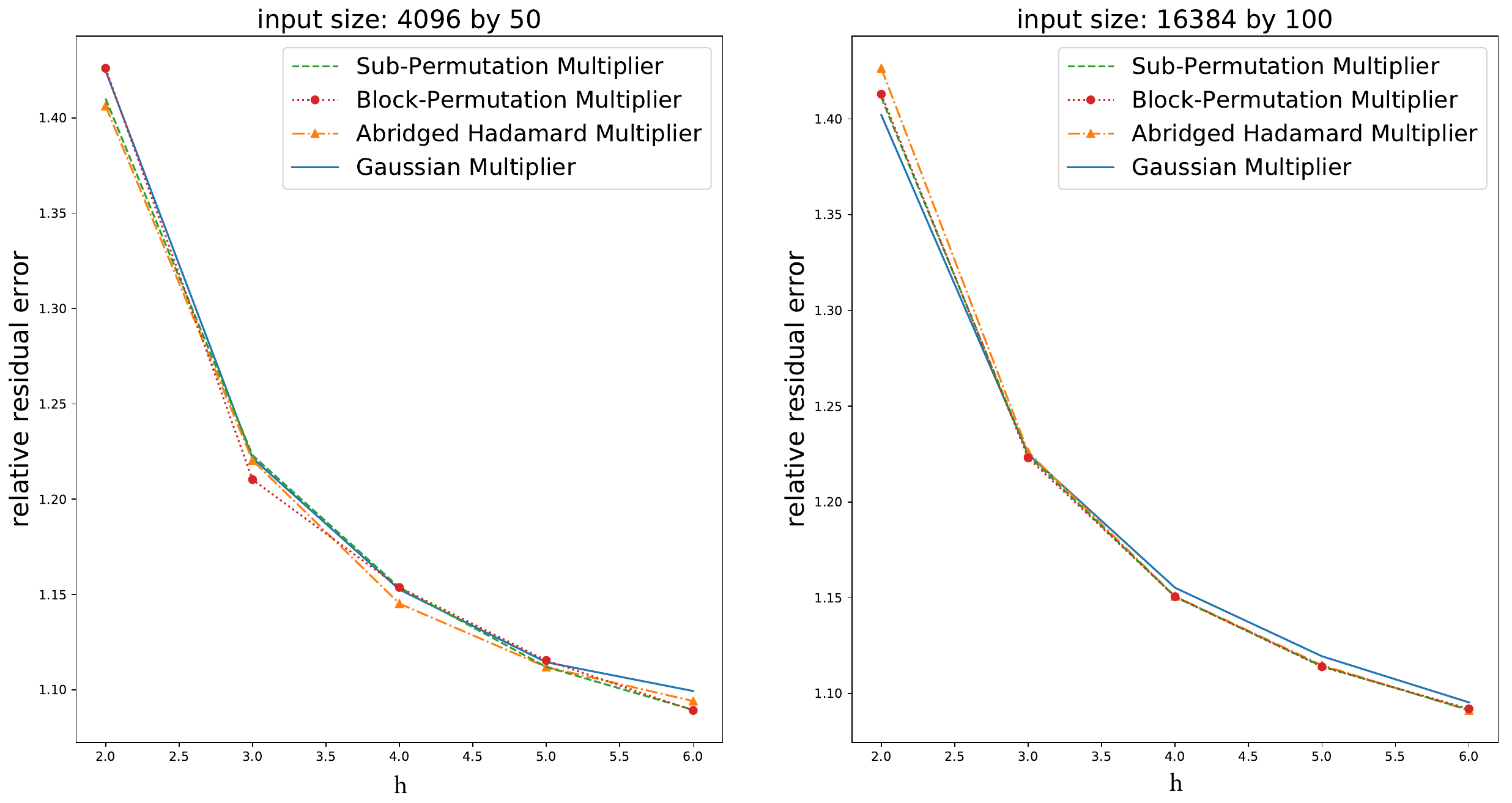}
\caption{Relative  residual norm in tests with Gaussian inputs}\label{LLSPRandom}
\end{figure}

\begin{figure}
\includegraphics[width = \textwidth]{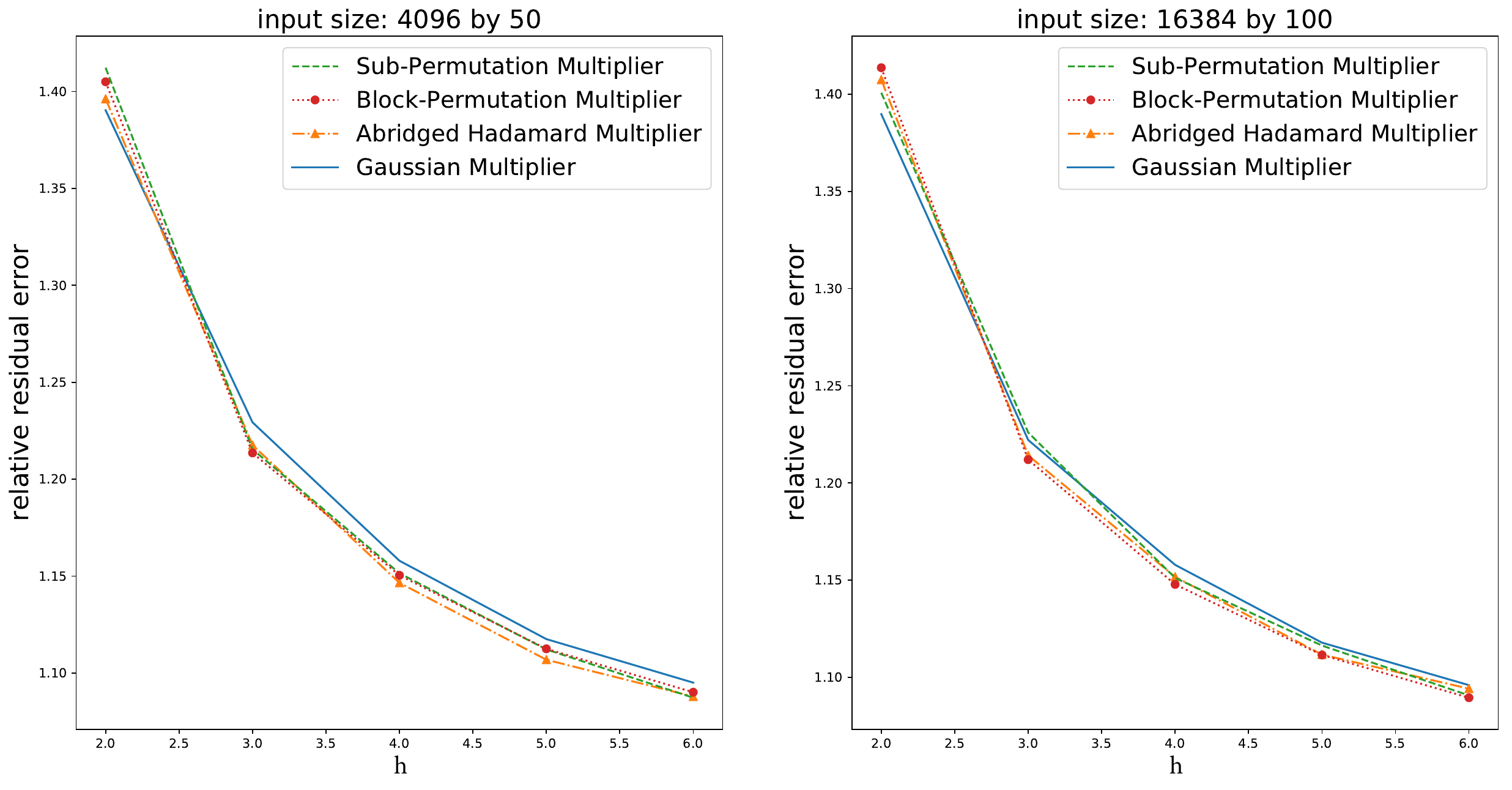}
\caption{Relative residual norm in tests with ill-conditioned random inputs}\label{LLSPill}
\end{figure}

\subsection{Red Wine Quality Data and California Housing Prices Data}

In this subsection we present the test results for real world inputs, namely the Red Wine Quality Data 
and California Housing Prices Data. For each triple of the datasets, multiplier type and multiplier size, 
we repeated the test for 100 random multipliers and computed the mean 
relative residual norm.
The results for these two 
input classes
are displayed in Figures \ref{LLSPRedWine} 
and  \ref{LLSPCaliHousing}.

\medskip

11 physiochemical feature data
of the Red Wine Quality Data  such as {\em fixed acidity},
{\em residual sugar level}, and {\em pH level}
 were the input  variables in our tests
and one sensory data {\em wine quality} were the output data;
the tests covered 1599 variants of the Portuguese "Vinho Verde" wine. 
See further information in \cite{CCAMR09}.  This dataset is often applied in regression tests that use physiochemical data of a specific wine 
in order to predict its quality, 
and among various types of regression
LLSP algorithms are considered  a popular choice.

From this dataset we constructed a $2048\times 12$ input matrix $A$ with each row representing one variant of red wine,
and with columns consisting of a bias column and  eleven physiochemical feature columns.
 The input vector ${\bf b}$ is a vector consisting of the 
wine quality level (between 0 and 10) for each variant.
 We kept the 1599 rows of the original data, 
 padded the rest of the rows with zeros, and performed a full row permutation of $A$. 

\medskip

The California Housing Prices data appeared in \cite{PB97} and were collected from the 1990 California Census,
including 9 attributes for each of the 20,640 Block Groups observed. This dataset is used for regression
tests in order to predict the {\em median housing value} of a certain area given collected information of this area, such as
{\em population}, {\em median income}, and {\em housing median age}. 

We randomly selected 16,384 observations from the dataset in order to construct an independent input matrix 
$A_0$ of size $16384\times 8$ and a dependent input vector ${\bf b} \in \mathbb{R}^{16384}$. Furthermore,
we augmented $A_0$ with a single bias column, i.e. 
$A = \begin{bmatrix}A_0 & {\bf 1}  \end{bmatrix} $.

\medskip

We computed approximate solutions by applying the algorithm supporting Theorem \ref{thLLSPdd} and using our multipliers.
 Figure \ref{LLSPRedWine} and \ref{LLSPCaliHousing} show that the resulting solution 
was almost as accurate as the optimal solution. Moreover, using Gaussian multipliers rather than  our sparse multipliers
only 
 enabled a marginal decrease of relative residual norm. 

\begin{figure}
\includegraphics[width = \textwidth]{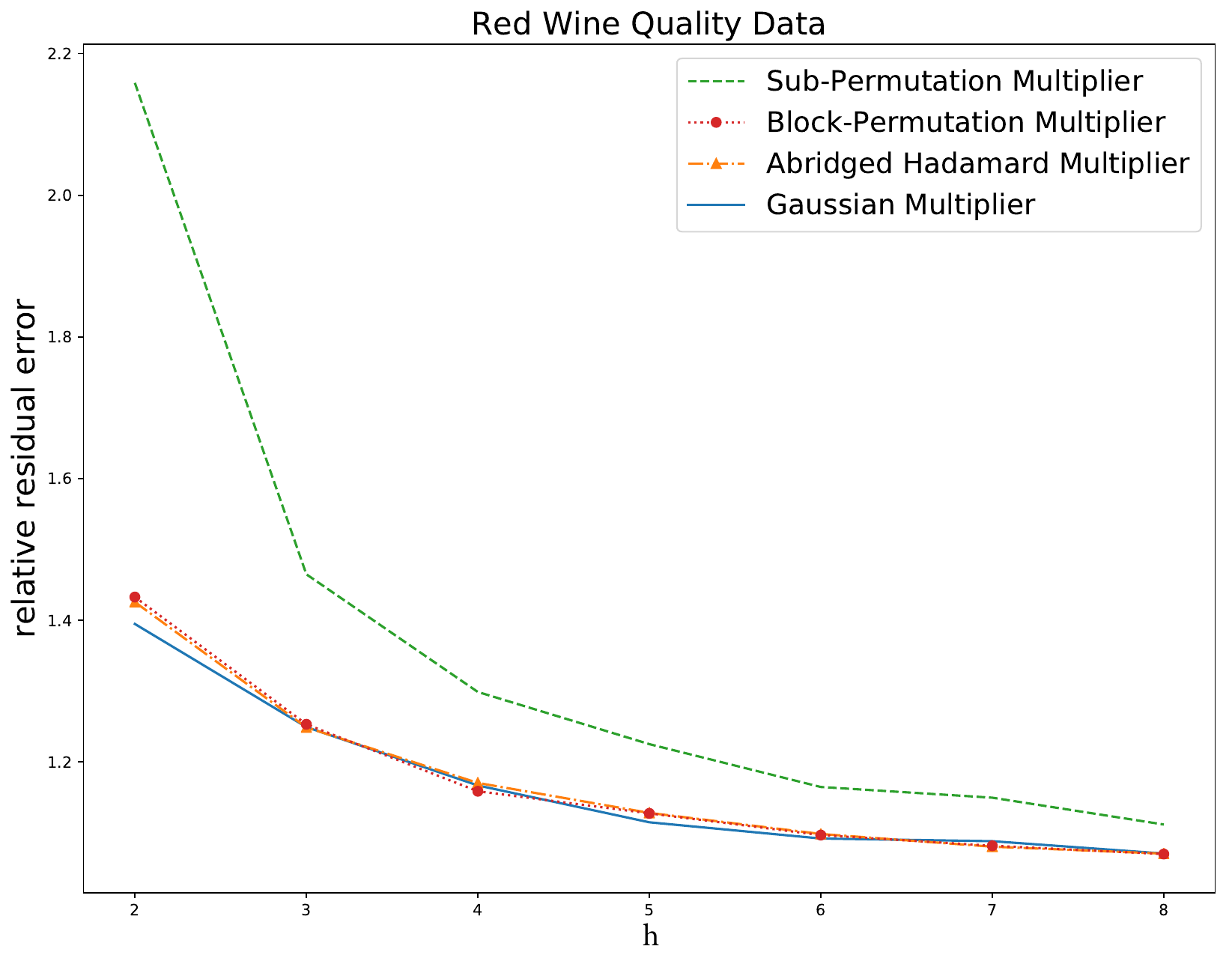}
\caption{Relative residual norm in tests with Red Wine Quality Data}\label{LLSPRedWine}
\end{figure}

\begin{figure}
\includegraphics[width = \textwidth]{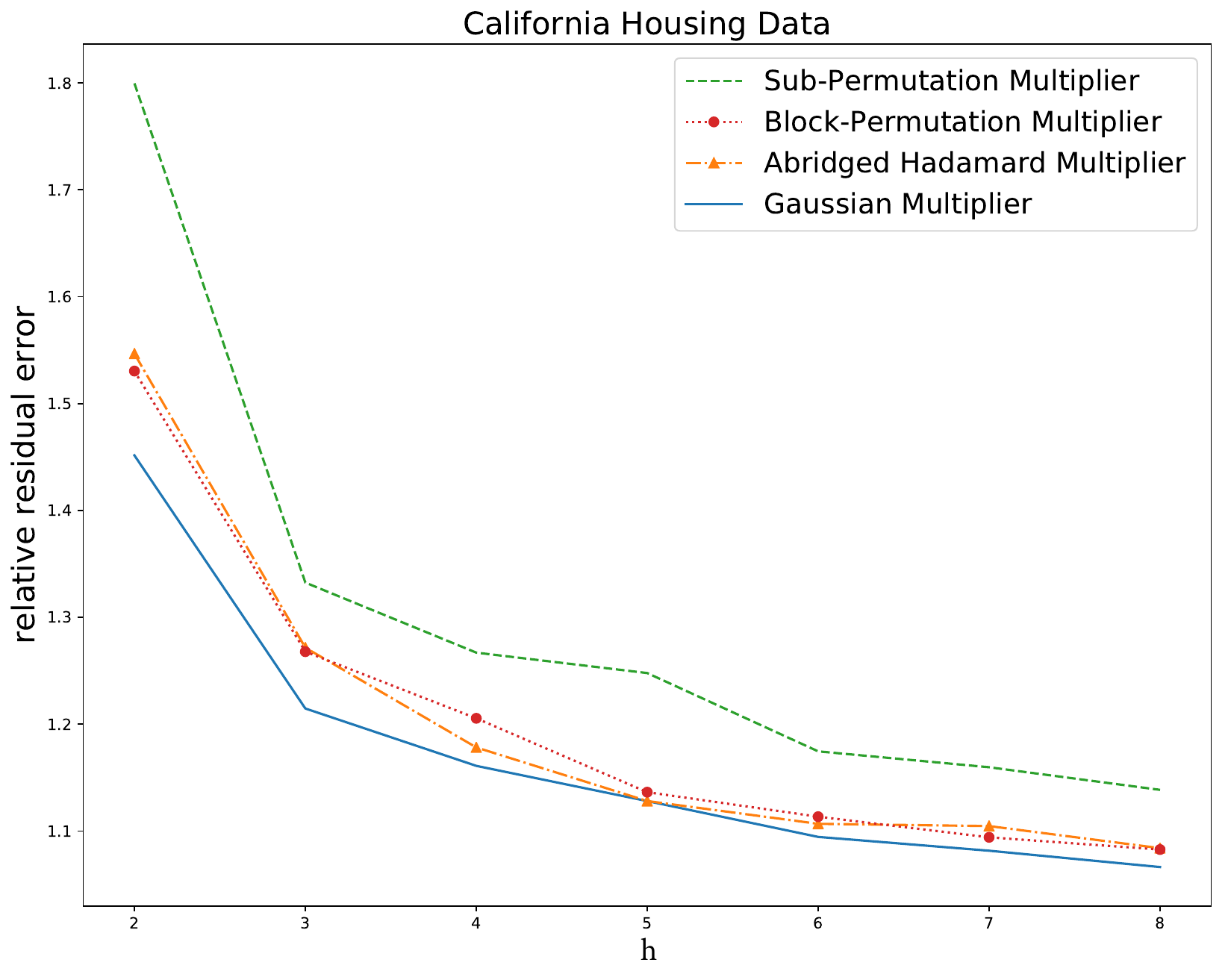}
\caption{Relative residual norm in tests with California Housing Prices Data}\label{LLSPCaliHousing}
\end{figure}







\bigskip


\noindent {\bf Acknowledgements:}
Our work was supported by NSF Grants 
 CCF--1563942 and CCF--1733834
and PSC CUNY Award 69813 00 48.



\end{document}